\documentclass[opre,nonblindrev]{informs3} 

\usepackage{caption}
\usepackage{anyfontsize}
\usepackage{subfiles}
\usepackage{xr}
\usepackage{float}
\def\biblio{\bibliographystyle{plainnat}\bibliography{../DataDrivenRobustBibliography}}

\usepackage{accents}

\usepackage{natbib}
 \bibpunct[, ]{(}{)}{,}{a}{}{,}%
 \def\bibfont{\small}%
\TheoremsNumberedThrough     
\ECRepeatTheorems
\EquationsNumberedThrough    

\RequirePackage[shortlabels]{enumitem}
\usepackage{mathtools} 
\usepackage{enumitem} 
\usepackage{dsfont} 
\usepackage{mathrsfs}  

\newtheorem{innercustomgeneric}{\customgenericname}
\providecommand{\customgenericname}{}
\newcommand{\newcustomtheorem}[2]{%
	\newenvironment{#1}[1]
	{%
		\renewcommand\customgenericname{#2}%
		\renewcommand\theinnercustomgeneric{##1}%
		\innercustomgeneric
	}
	{\endinnercustomgeneric}
}

\newcustomtheorem{customthm}{Theorem}
\newcustomtheorem{customlemma}{Lemma}

\allowdisplaybreaks

\usepackage{epsfig}

\usepackage[ruled, vlined, linesnumbered]{algorithm2e}
\let\oldnl\nl
\newcommand{\nonl}{\renewcommand{\nl}{\let\nl\oldnl}}

\usepackage{IEEEtrantools}
\usepackage{multicol}
\usepackage{caption}
\usepackage{subcaption}

\usepackage[margin=2.54cm]{geometry}
\usepackage{lipsum}
\usepackage{bigints}
\usepackage{framed}
\usepackage{colortbl}
\usepackage{calligra}
\usepackage[export]{adjustbox}
\usepackage{booktabs}
\usepackage{multirow}
\usepackage{makecell}

 \usepackage{wrapfig}

\usepackage{bbm}
\usepackage[normalem]{ulem}

\usepackage[colorlinks=true,linkcolor=black,citecolor=black]{hyperref}
\usepackage{cleveref}

\newcommand{\R}{{\mathcal R}}

\def\RR{ {\mathbb{R}}}

\def \bE {\mathbb{E}}
\def \bR {\mathbb{R}}

\newcommand{\rr}[1]{{\textcolor{red}{[RC: #1]}}}

\let\R\Real

\renewcommand{\qed}{{\hfill\halmos}}

\def \bE {\mathbb{E}}
\def \bR {\mathbb{R}}

\usepackage{graphicx}
\usepackage{tikz,ifthen}

   \usetikzlibrary{math}
   \usetikzlibrary{shapes.misc}
   \usetikzlibrary{shapes.multipart,positioning,patterns,backgrounds}

\usepackage{multirow}
\usepackage{array}
\usepackage{cleveref} 

\usepackage{booktabs}
\usepackage{multirow}
\usepackage{makecell}

\usepackage{array}
\newcolumntype{x}[1]{>{\centering\arraybackslash}p{#1}}

\begin{document}

\def\biblio{}




\RUNTITLE{Smooth Learning via Legendre-Regularized Policies}

\ARTICLEAUTHORS{Zikun Lin\thanks{Department of Systems Engineering and Engineering Management, The Chinese University of Hong Kong, Hong Kong, China. Email: {\tt zklin@se.cuhk.edu.hk}.},  
Rui Chen\thanks{School of Data Science, The Chinese University of Hong Kong, Shenzhen, Guangdong, China. Email: {\tt rchen@cuhk.edu.cn}.},
and
Yijie Wang\thanks{School of Economics and Management, 
Tongji University, Shanghai, China. Email: {\tt yijiewang@tongji.edu.cn}.}
}

\TITLE{Smooth Learning with Hard Constraints via Legendre-Regularized Policies}

\ABSTRACT{
We revisit contextual optimization from the perspective of policy class design. A desirable policy class should be expressive enough to learn rich context-decision relationships, should enforce hard feasibility constraints rather than soft penalty terms, and should remain smooth enough for gradient-based training on downstream decision losses. Existing approaches usually emphasize only part of these requirements. We propose Legendre-regularized policies, which parameterize decisions as solutions of regularized optimization problems over the original feasible region. This construction yields policies that are feasible by construction and differentiable with respect to learned latent parameters. We prove that the associated optimizer map is single-valued, maps onto the relative interior of the feasible set, admits an explicit Jacobian, is Lipschitz continuous, and can be made arbitrarily smooth. We also establish a universal approximation result showing that the proposed class can approximate any continuous feasible policy on compact context sets. The framework unifies explicitly regularized optimizers and implicit perturbation-based smooth optimizers. Experiments on contextual newsvendor and resource allocation problems show that our approach improves prescriptive performance relative to the benchmark methods.
}
\maketitle

\section{Introduction}
\label{sec:introduction}

A central problem in data-driven decision-making is to learn a policy that maps contextual information to a feasible decision with low downstream cost. In population form, this problem can be written as
\begin{align}\label{generic_CO}
    \min_{\pi\in \Pi} \bE_{(x,\xi)} \ell(\pi(x);\xi),
\end{align}
where $x \in \RR^{k_1}$ denotes the side information that is available before making the decision, $\xi \in \RR^{k_2}$ denotes the uncertain parameter governing downstream performance, $\ell(\cdot;\xi):S\rightarrow \bR$ is the cost function, $S\subseteq \bR^{n}$ is a set representing the deterministic feasible region, and $\Pi$ is a class of policies mapping each context $x$ to a feasible decision $w \in  S$. The key question is therefore how to design the policy class in \eqref{generic_CO}. A desirable policy class must satisfy three requirements at once: it must be expressive enough to capture the dependence of good decisions on context, it must enforce the hard constraints (as opposed to soft constraints or penalized violations) defining $S$, and it must be smooth enough for gradient-based training. 

Existing approaches emphasize different parts of these requirements. The decision rule approach seeks to learn a direct mapping from contextual covariates to optimal decisions, optimizing an empirical risk metric over a predefined hypothesis class~\citep{brandt2009parametric,ban2019big}. This approach is natural from a policy-learning perspective, but feasibility is difficult to guarantee for flexible models when $S$ is high-dimensional or polyhedral. Predict-then-optimize methods take the opposite route: they preserve the downstream optimization model, but train the statistical model for prediction rather than for the decision loss~\citep{bertsimas2020predictive,bertsimas2023dynamic,kannan2024residuals,kannan2025data,wang2026data}. Decision-focused learning attempts to combine both advantages by training the predictive model through the downstream decision loss~\citep{amos2017optnet,agrawal2019differentiable}. However, when the decision is obtained by solving an optimization problem, the induced optimizer map is often non-smooth or set-valued. For example, in linear optimization over a polyhedral feasible region, the map from predicted costs to optimal decisions is piecewise constant, leading to zero or undefined gradients in large regions of the parameter space~\citep{bolte2021nonsmooth}. This has motivated surrogate losses and differentiable optimization layers~\citep{berthet2020learning,elmachtoub2022smart}, as well as hard-constrained feasibility maps that project or reparameterize raw model outputs into the feasible set~\citep{chen2023end,schneider2026soft}.

This paper proposes a policy class that meets the three requirements by making the optimizer itself smooth. We study \emph{Legendre-regularized policies} of the form
\[
    \pi_g(x)=w_{F,\phi}(g(x)),\qquad
    w_{F,\phi}(z)\in\arg\min_{w\in S}\{\langle Fz,w\rangle+\phi(w)\}.
\]
Here $g$ is a learned model, and $\phi$ is a fixed Legendre-type regularizer chosen so that the resulting optimizer is well-defined and smooth. We formalize the required conditions on $\phi$ through the notion of an \emph{admissible Legendre regularizer} in Section~\ref{subsec:regularizer}. The output of $g$ is not an infeasible candidate decision to be repaired, nor does it need to be interpreted as a prediction of $\xi$. Instead, it serves as latent intermediate parameters, which are coefficients of a regularized optimization problem whose solution is the prescribed decision.

The resulting solution map has the desired property for learning. Under mild conditions (defining admissible Legendre regularizers), it is single-valued, lies in the relative interior $S^\circ$, is differentiable with an explicit Jacobian, and is Lipschitz continuous in the latent parameters. With smoother regularizers, the policy map can inherit higher-order smoothness, which gives continuous and stable sensitivities of the decision with respect to the learned latent parameters. This is stronger than the almost-everywhere differentiability available from many projection-based maps: first-order methods can often still be applied in the latter case, but gradients may not exist and may change discontinuously across active regions~\citep{schneider2026soft}. In contrast, the Jacobian of $w_{F,\phi}$ is controlled by the regularizer through the Hessian of its convex conjugate, yielding a more structured object for backpropagation, sensitivity analysis, and higher-order optimization methods. At the same time, one can show that parameterizing decisions as solutions of regularized optimization problems does not sacrifice expressive power: since the image of $w_{F,\phi}$ is $S^\circ$, composing it with a universal approximator yields a policy class that can uniformly approximate any continuous feasible policy on a compact context set.

The framework also unifies several smoothing mechanisms. Standard choices, such as logarithmic or entropic regularization, give admissible Legendre regularizers for the feasible regions, and lead to optimization layers with tractable derivatives.  Meanwhile, perturbation-based smooth optimizers~\citep{berthet2020learning} can also be interpreted as \emph{implicit} Legendre-regularized policies. This extends the framework beyond analytically specified regularizers to oracle-based smoothing schemes, which is useful when an explicit regularization is inconvenient or inefficient but an optimization oracle over $S$ is available. 


The contribution of our paper can be summarized as follows:
\begin{enumerate}
    \item We propose a Legendre-regularized policy framework for decision-focused learning. We show that the induced solution map is single-valued, feasible by construction, continuously differentiable with an explicit Jacobian, Lipschitz continuous, and maps onto the relative interior of $S$. We also show that our Legendre-type assumptions on the regularizer are necessary for these desirable properties of the corresponding solution map.
    \item We provide a unified framework that integrates existing explicit and implicit smooth optimizers. In particular, both analytically specified Legendre regularizers and perturbation-induced regularizers can be treated as instances of our proposed Legendre-regularized policies.
    \item We show that this Legendre-regularized policy framework preserves expressive power: when the latent model class $\mathcal G$ is a universal approximator, Legendre-regularized policies can uniformly approximate any continuous feasible policy on a compact context set.
    \item We conduct extensive experiments on contextual newsvendor and resource allocation problems to demonstrate the effectiveness of the proposed framework.
\end{enumerate}

\noindent \textbf{Notations}.
For a set $S$, we denote by $\operatorname{aff}(S)$ its affine hull, $\operatorname{cl}(S)$ its closure, $S^\circ$ its relative interior, and $\operatorname{int}(S)$ its interior when the ambient space is clear, and we denote by $\mathbb I_S(\cdot)$ and $\mathds{1}_S(\cdot)$ its associated extended-valued and 0-1 indicator functions, respectively, defined as $\mathbb I_S(w)=0$ and $\mathds{1}_S(w)=1$ if $w\in S$, and $\mathbb I_S(w)=+\infty$ and $\mathds{1}_S(w)=0$ otherwise. For an extended-real-valued function $\phi:\R^n\rightarrow\R\cup\{-\infty,+\infty\}$, $\operatorname{dom}(\phi)$ denotes its effective domain, i.e., $\operatorname{dom}(\phi):=\{w\in\R^n:\phi(w)<+\infty\}$, and $\phi^*$ denotes its convex conjugate, i.e., $\phi^*(\cdot):=\sup_w \{\langle \cdot,w\rangle-\phi(w)\}$. 
For an integer $k\geq 1$ and a set $X$, $C^k(X)$ denotes the class of functions with continuous (partial) derivatives up to order $k$ on $X$, and $C^\infty(X)$ denotes smooth functions with continuous (partial) derivatives of all orders on $X$. Let $\mathcal{N}(\mu,\sigma^2)$ denote the Gaussian distribution with mean $\mu$ and variance $\sigma^2$, and $\Phi(\cdot)$ denote the cumulative distribution function of $\mathcal{N}(0,1)$.


\subsection{Related Literature}

\paragraph{Policy learning.} A large literature studies data-driven policies that map covariates directly to decisions, including parametric decision rules~\citep{brandt2009parametric,ban2019big,bazier2020generalization}, piecewise-affine decision rules~\citep{zhang2026data}, kernel and forest-based prescriptive methods~\citep{bertsimas2022data,kallus2023stochastic}, and neural decision rules~\citep{zhang2020deep,oroojlooyjadid2020applying,qi2023practical}. However, policy learning methods do not automatically ensure feasibility; therefore, a recent line of work aims to enforce feasibility by composing an unconstrained predictor with a projection or radial map onto the feasible set~\citep{chen2023end,schneider2026soft}. Our construction shares this goal but uses a different mechanism: instead of geometrically correcting a raw decision, we parameterize feasible decisions as solutions of Legendre-regularized optimization problems. This optimizer-based parameterization yields Jacobians governed by the regularizer through convex duality. Moreover, with sufficiently smooth Legendre regularizers, the resulting policy map can be made smooth to higher order, and in standard smooth-regularizer cases even $C^\infty$, which is stronger than the almost-everywhere differentiability typical of projection-based maps over polyhedral sets.

\paragraph{End-to-end learning.} 
End-to-end learning, also known as decision-focused learning, trains predictive models through the downstream optimization problem so that model parameters are chosen for decision quality rather than prediction accuracy alone~\citep{donti2017task,wilder2019melding,elmachtoub2022smart}. A common approach is to differentiate through an optimization layer, as in differentiable quadratic or convex optimization layers~\citep{amos2017optnet,agrawal2019differentiable}. This strategy is powerful when the original solution map is smooth; however, many decision problems have non-smooth or set-valued optimizers. In linear optimization over a polyhedral feasible region, for example, the map from predicted costs to optimal decisions is piecewise constant and provides little useful gradient information. 
Existing work addresses this issue through surrogate losses such as SPO+ and related decision-aware objectives~\citep{elmachtoub2020decision,mandi2020smart,elmachtoub2022smart,loke2022decision,liu2021risk}, or through perturbed and Fenchel--Young smooth optimizers~\citep{berthet2020learning}. Many such surrogates are tailored to particular downstream structures, most notably linear objectives in the Smart Predict-then-Optimize (SPO) framework \citep{mandi2020smart, elmachtoub2022smart}. Our approach shifts the focus from designing a surrogate loss for a fixed non-smooth optimizer to designing a smooth feasible policy class directly. Because the smoothing is applied to the decision map rather than to a problem-specific loss surrogate, Legendre-regularized policies can be trained on general (sub)differentiable decision losses.

\paragraph{Regularization as smoothing.}
Regularization is a widely used mechanism for smoothing non-smooth optimization problems. For non-smooth convex functions that can be represented as pointwise maxima, \cite{nesterov2005smooth} shows that adding a strongly convex function to the inner optimization problem yields a smooth approximation with a Lipschitz-continuous gradient. Related regularization-based smoothing techniques have also been studied in online linear optimization~\citep{abernethy2014online}. In structured prediction, SparseMAP adds a quadratic regularizer to maximum a posteriori (MAP) inference to obtain a unique solution map that is differentiable almost everywhere~\citep{niculae2018sparsemap}, while the Fenchel--Young framework provides a general treatment of prediction maps induced by convex regularizers~\citep{blondel2020learning}. 
In decision-focused learning, quadratic and $\ell_2$-norm regularizers have been used to smooth linear optimizer maps~\citep{wilder2019melding,wang2026machine}. These methods address their zero or undefined gradients and enable differentiation through the KKT conditions. Our work builds on this regularization-as-smoothing principle but focuses on the design of the resulting policy class.  In particular, Legendre regularization ensures that every latent parameter produces a solution in the relative interior of the feasible region. The resulting solution map covers the entire relative interior and is differentiable everywhere, and can be made arbitrarily smooth under stronger regularity conditions.

\section{Legendre-Regularized Policies}\label{sec:legendre_regularized_policies}

We now introduce a class of contextual decision policies that are feasible by construction and differentiable with respect to learned parameters. The basic idea is to avoid learning decisions directly. Instead, the learned model outputs parameters of a regularized optimization problem, and the policy is obtained by solving this optimization problem over the original feasible region.

Let $g:\mathcal X\to\mathbb R^k$ be a prediction model, where $\mathcal X$ denotes the context space. Given a fixed feasible region $S$, a classical plug-in policy for a linear downstream objective has the form
\[
x\mapsto w^\star(g(x)),
\qquad
w^\star(z)\in \arg\min_{w\in S} z^\top w.
\]
When $S$ has a non-smooth boundary (e.g., $S$ is polyhedral), the map $z\mapsto w^\star(z)$ is generally set-valued and non-smooth. This creates difficulties for gradient-based end-to-end training. Our approach replaces the non-smooth optimizer $w^\star(\cdot)$ with a smooth regularized optimizer.

Throughout this section, assume that the deterministic feasible region $S$ is a nonempty closed convex set with its affine hull aff$(S)=\{w\in\R^n:Aw=b\}$, where $A\in\R^{p\times n}$ has full row rank and $0\leq p\leq n-1$. Note that, in the case when $S$ is full dimensional, aff$(S)=\R^n$, $A$ is absent, and $p=0$. 

Let $r=n-p$, and let $N\in\mathbb R^{n\times r}$ be a matrix whose columns form an orthonormal basis of $\ker(A)$. Fix any point $\bar w$ in the relative interior $S^\circ$ of $S$. Then every feasible point $w\in$aff$(S)$ can be written uniquely as
$w=\bar w+Ny$,
for some $y\in \R^r$.
Define the reduced feasible region
\[
Y:=\{y\in\mathbb R^r:\bar w+Ny\in S^\circ\}.
\]
Then $Y$ provides a representation of $S^\circ$ up to an affine transformation.

\subsection{Admissible Legendre Regularizers}\label{subsec:regularizer}

We impose the following regularity conditions on the regularizer.

\begin{definition}
	A proper lower semicontinuous convex function $\phi:\mathbb R^n\to \mathbb R\cup\{+\infty\}$
	is called an \emph{admissible Legendre regularizer} for $S$ if $S^\circ=(\operatorname{dom}(\phi))^\circ$, and the reduced function
	\[
	\psi(y):=\phi(\bar w+Ny)
	\]
	satisfies the following conditions:
	\begin{enumerate}
		\item $\psi$ is $\mu$-strongly convex for some $\mu>0$;
		\item $\psi$ is essentially smooth (i.e., $\psi$ is differentiable on $Y$, and $\lim_{i\rightarrow \infty}|\nabla \psi(y^i)|=+\infty$ for all sequences $\{y^i\}_{i=1}^\infty\subseteq Y$ converging to any boundary point of $Y$);
		\item Its convex conjugate $\psi^*$ is twice differentiable on $\mathbb R^r$.
	\end{enumerate}
    \label{def:Legendre}
\end{definition}
In the full-dimensional case when $A$ is absent ($n=r$, $N=I_n$), conditions 1-3 reduce to the simpler conditions that $\phi$ is $\mu$-strongly convex, essentially smooth, and has a twice continuously differentiable convex conjugate. Note that admissible Legendre regularizers are \emph{convex functions of Legendre type} \citep{rockafellar1970convex} due to conditions 1 and 2.

Below we summarize some useful properties of admissible Legendre regularizers.
\begin{lemma}\label{lem:basic}
    Let $\psi:\mathbb R^r\to \mathbb R\cup\{+\infty\}$ be a $\mu$-strongly convex and essentially smooth function with $Y=\operatorname{int}(\operatorname{dom}(\psi))$. Then the following properties hold:\begin{enumerate}
        \item $\psi^*:\R^r\rightarrow\R$ is differentiable;
        \item $\nabla\psi^*:\R^r\rightarrow Y$ is $1/\mu$-Lipschitz continuous;
        \item $\nabla \psi$ is continuous and bijective from $Y$ to $\R^r$ with $(\nabla \psi)^{-1}=\nabla \psi^*$;
        \item For all $k\in\mathbb Z_+$, $\psi^*\in C^k(\R^r)$ if $\psi\in C^k(Y)$.
    \end{enumerate}
\end{lemma}
\begin{proof}{Proof.}
    Properties 1-3 follow from standard convex analysis (\citealt{rockafellar1970convex}[Theorem 25.5, Theorem 26.5], \citealt{rockafellar2009variational}[Proposition 12.60]). Note that $\nabla^2 \psi(y)$ is positive definite (and therefore non-singular) for all $y\in Y$ due to strong convexity of $\psi$. Then property 4 follows from property 3 and the inverse function theorem.\qed
\end{proof}
As a special case of property 4 in Lemma \ref{lem:basic}, a sufficient (but not necessary) condition for condition~3 in Definition \ref{def:Legendre} is $\psi\in C^2(Y)$, provided that the remaining conditions in Definition \ref{def:Legendre} are satisfied.


\subsection{Legendre-Regularized Policy Class}

Let $F\in\mathbb R^{n\times k}$ with $k\geq r$, and define $B:=N^\top F\in\mathbb R^{r\times k}$.
We assume that $\operatorname{rank}(B)=r$ (e.g., $F=I_n$ in which case $k=n$ and $B=N^\top$, or $F=N$ in which case $k=r$ and $B=I_r$).


Given an admissible Legendre regularizer $\phi$, define the regularized optimization layer
\begin{equation}\label{layer:opt}
    w_{F,\phi}(z)
\in
\arg\min_{w\in\mathbb R^n}
\left\{
\langle Fz,w\rangle+\phi(w)
\right\}.
\end{equation}
Since $\phi$ has its effective domain contained in $S$, this is equivalently a strongly convex optimization problem over the original feasible region $S$. A \textit{Legendre-regularized policy} (LRP) is then a composition
\[
\pi_g(x):=w_{F,\phi}(g(x)),
\qquad g\in\mathcal G,
\]
where $\mathcal G$ is a chosen class of prediction models, such as neural networks, that map side information to a latent parameter vector in $\R^k$. The induced policy class is
\[
\Pi_{F,\phi}(\mathcal{G})
:=
\{w_{F,\phi}\circ g:g\in\mathcal G\}.
\]
This policy class differs from the class of plug-in policies (i.e., $\{w^{\star}\circ g:g\in\mathcal G\}$) in an important way. For an LRP $\pi_g$, the model output $g(x)$ need not be interpreted as a prediction of the uncertain objective coefficients in a linear optimization problem over $S$. Instead, $g(x)$ parameterizes a regularized optimization problem whose solution is the decision. Thus, the downstream loss used for training can be any (sub)differentiable decision loss $\ell(w;\xi)$, not only a linear cost.

\subsection{Examples}\label{sec:examples}
We provide a few examples of admissible Legendre regularizers in the case when $S$ is a polytope (i.e., a bounded polyhedron) of the form $S=\{w:Aw=b,Cw\geq d\}$ with $C\in\R^{m\times n}$ and $Cw\geq d$ being facet-defining for $S$, in which case $S^\circ=\{w:Aw=b,Cw>d\}$.

\subsubsection{$\mathsf{LRP}$-Log.} A classic example is the logarithmic barrier
\begin{equation}\label{regularizer:log}
    \phi_{\log}(w)
=
-\tau\sum_{j=1}^m \log(C_jw-d_j)
+\mathbb I_{\{w:Aw=b\}}(w),
\end{equation}
where $\tau>0$, $C_j$ denotes the $j$-th row of $C$, the function is extended by $+\infty$ whenever
$C_jw-d_j\leq 0$ for some $j$, and $\mathbb I_{\{w:Aw=b\}}:\R^n\rightarrow\{0,+\infty\}$. Note that the indicator function can be easily encoded as constraints in the solution of \eqref{layer:opt}. 

\subsubsection{$\mathsf{LRP}$-Ent.} Another useful example is the entropic regularizer
\begin{equation}\label{regularizer:entropy}
    \phi_{\mathrm{ent}}(w)
    =
    \tau\sum_{j=1}^m
    (C_jw-d_j)\log(C_jw-d_j)
    +
    \mathbb I_{\{w:Aw=b\}}(w),
\end{equation}
where $\tau>0$, $0\log 0:=0$, and the function is extended by $+\infty$ whenever
$C_jw-d_j< 0$ for some $j$. Unlike the logarithmic barrier, this regularizer does not diverge
in value at the boundary. However, its gradient diverges as any slack approaches zero, and hence
the corresponding $\psi$ is essentially smooth.

\subsubsection{$\mathsf{LRP}$-Ptb.} Moreover, regularizers can be made implicit. We next show an example of an admissible Legendre regularizer that does not require an algebraic description of $S$, but a linear optimization oracle over $S$. Given a full-dimensional polyhedral $S$ and some well-behaved noise vector $Z$ (e.g., $Z\sim\mathcal N(0,\tau^2\cdot I_n)$), \cite{berthet2020learning} show that the following perturbed minimum
\[
w_{\mathrm{ptb}}(z):=\mathbb E_Z\left[w_{\mathrm{ptb}}^Z(z)\right]~\textrm{with }w_{\mathrm{ptb}}^Z(z)\in\arg\min_{w\in S}\langle z+ Z,w\rangle
\]
can be written as a regularized minimum
\[
w_{\mathrm{ptb}}(z)\in\arg\min_{w\in S}\left[\langle z,w\rangle+\phi_{\mathrm{ptb}}(w)\right],
\]
with $\phi_{\mathrm{ptb}}$ being the convex conjugate of $F(z)=\bE_Z[\max_{w\in S}\langle z-Z,w\rangle]$.
Furthermore, under the assumptions in \cite{berthet2020learning}, $\phi_{\mathrm{ptb}}$ is strongly convex and Legendre-type (implying essential smoothness) with its convex conjugate $\phi_{\mathrm{ptb}}^*=F$ being twice differentiable. 
Therefore, $w_{\mathrm{ptb}}=w_{I_n,\phi_{\mathrm{ptb}}}$ is an LRP with $\phi_{\mathrm{ptb}}$ being an admissible Legendre-regularizer implicitly defined as the convex conjugate of an expected parametric maximum. Nevertheless, stochastic gradient/Jacobian information can be evaluated by Monte-Carlo methods \citep{berthet2020learning}.

\subsubsection{Unbounded $S$.} In the case when the feasible region $S=\{w:Aw=b,Cw\geq d\}$ is unbounded, the logarithmic barrier \eqref{regularizer:log} or the entropic regularizer \eqref{regularizer:entropy} no longer implies a strongly convex $\psi$. But one may simply add to them a strongly convex real-valued differentiable function (e.g., $\|w\|_2^2$) so that the corresponding $\psi$ becomes strongly convex, and then the augmented $\phi$ becomes an admissible Legendre regularizer.



\subsection{Resulting Solution Map}

The following result summarizes the basic properties of the resulting solution map.

\begin{theorem}
	\label{thm:main}
    Suppose $\phi$ is an admissible Legendre regularizer for $S$. Then the following hold:
	
	\begin{enumerate}
		\item For every $z\in\mathbb R^k$, the optimization problem \eqref{layer:opt} defining $w_{F,\phi}(z)$ has a unique optimal solution, and the unique optimal solution $w_{F,\phi}(z)\in S^\circ$;
		
		\item The solution map admits the representation
		\[
		w_{F,\phi}(z)
		=
		\bar w
		+
		N\nabla\psi^*(-Bz).
		\]
		Consequently, $w_{F,\phi}$ is differentiable, with Jacobian
		\[
		Jw_{F,\phi}(z)
		=
		-N\nabla^2\psi^*(-Bz)B;
		\]
		
		\item The map $w_{F,\phi}$ is Lipschitz continuous. In particular, for $z_1,z_2\in\R^k$,
		\[
		\|w_{F,\phi}(z_1)-w_{F,\phi}(z_2)\|_2
		\le
		\frac{\|B\|_2}{\mu}
		\|z_1-z_2\|_2
		\le
		\frac{\|F\|_2}{\mu}
		\|z_1-z_2\|_2;
		\]
		
		\item The image of $w_{F,\phi}$ is $S^\circ$. If $k=r$, then $w_{F,\phi}$ is bijective from $\mathbb R^r$ onto $S^\circ$, and its inverse $w_{F,\phi}^{-1}$ is continuous and satisfies
		\[
		w_{F,\phi}^{-1}(\bar w+Ny)
		=
		-B^{-1}\nabla\psi(y).
		\]
		If, in addition, $\psi\in C^{k}(Y)$ for some integer $k\geq 2$, then $w_{F,\phi}\in C^{k-1}(\R^r)$ and $w^{-1}_{F,\phi}\in C^{k-1}(S^\circ)$.
	\end{enumerate}
\end{theorem}

\begin{proof}{Proof.}
	Using the reduced representation $w=\bar w+Ny$, the regularized optimization problem becomes $\min_{y}
	\left\{
	\langle Fz,\bar w+Ny\rangle+\psi(y)
	\right\}$,
	which is equivalent to
    \begin{equation}\label{reduced_problem}
        \min_y
	\left\{
	\langle Bz,y\rangle+\psi(y)
	\right\}.
    \end{equation}
	Since $\psi$ is strongly convex and essentially smooth, this problem has a unique minimizer which lies in $\operatorname{int}(\operatorname{dom}\psi)=Y$. Therefore $w_{F,\phi}(z)\in S^\circ$.
	
	The first-order condition of \eqref{reduced_problem} is $Bz+\nabla\psi(y)=0$, which is equivalent to $-Bz=\nabla\psi(y)$.
	By property 3 of Lemma \ref{lem:basic}, $
	y=\nabla\psi^*(-Bz)$.
	Therefore,
	\[
	w_{F,\phi}(z)
	=
	\bar w+N\nabla\psi^*(-Bz).
	\]
	Since $\psi^*$ is twice differentiable, differentiating the above expression gives
	\[
	Jw_{F,\phi}(z)
	=
	-N\nabla^2\psi^*(-Bz)B.
	\]
	
	Due to property 2 of Lemma \ref{lem:basic},
	\[
	\begin{aligned}
		\|w_{F,\phi}(z_1)-w_{F,\phi}(z_2)\|_2
		=
		\|N[
		\nabla\psi^*(-Bz_1)-\nabla\psi^*(-Bz_2)
		]\|_2
		\le
		\frac{\|N\|_2\cdot\|B\|_2}{\mu}\|z_1-z_2\|_2.
	\end{aligned}
	\]
	Since $N$ has orthonormal columns, $\|N\|_2=1$ and $\|B\|_2=\|N^\top F\|_2\le \|F\|_2$.
	
	Finally, let $w=\bar w+Ny$ be an arbitrary point in $S^\circ$. Since $B$ has full row rank, there exists $z\in\mathbb R^k$ such that $Bz=-\nabla\psi(y)$.
	Then $y$ satisfies the first-order condition for the reduced problem \eqref{reduced_problem}, so $w=w_{F,\phi}(z)$. Thus the image of $w_{F,\phi}$ is $S^\circ$. If $k=r$ and $B$ is non-singular, this $z$ is unique and is given by $z=-B^{-1}\nabla\psi(y)$.
	The stated continuity and differentiability conclusions follow from the corresponding smoothness of $\nabla\psi$ and property 4 of Lemma \ref{lem:basic}.\qed
\end{proof}

Furthermore, one can show that the assumptions we made for an admissible Legendre regularizer $\phi$ (assumptions 1-3 in Definition \ref{def:Legendre}) are necessary for some properties in Theorem \ref{thm:main} to hold. For simplicity, we assume that $Y=S^\circ$ is bounded and full-dimensional, $n=r$, and $N=F=I_n$.
\begin{proposition}\label{prop:necessary_conditions}
    Let $Y\subseteq\R^n$ be a nonempty bounded open convex set, and $\psi:\R^r\rightarrow\R\cup\{+\infty\}$ be a proper lower semicontinuous convex function with $Y=\operatorname{int}(\operatorname{dom}(\psi))$. For $z\in\R^n$, define the solution map\begin{align}
        \bar w_{\psi}(z)\in\arg\min_{w\in\R^n}\{\langle z,w\rangle+\psi(w)\}.\label{define_wbar}
    \end{align}
    Assume the following properties of $\bar w_{\psi}$:\begin{enumerate}[label=(\roman*)]
        \item For all $z\in\R^n$, $\bar w_\psi(z)$ is uniquely defined, i.e., the right-hand side of \eqref{define_wbar} always has a unique minimizer;
        \item The map $\bar w_{\psi}:\R^n\rightarrow Y$ is bijective and $L$-Lipschitz continuous for some $L>0$;
        \item The map $\bar w_{\psi}$ is differentiable everywhere.
    \end{enumerate}
    Then function $\psi$ is $\mu$-strongly convex with $\mu=1/L$, essentially smooth, and $\psi^*$ is twice differentiable on $\R^n$.
\end{proposition}
\begin{proof}{Proof.}
    We first show that $\psi$ is differentiable on $Y$ and $\partial\psi(y)=\emptyset$ for $y\notin Y$. Let $y\in \R^n$ with $\partial\psi(y)\neq\emptyset$, $u,v\in\partial\psi(y)$ be two arbitrary subgradients of $\psi$ at $y$. Consider problem \eqref{define_wbar} with $z=-u$ and $z=-v$. Then, by the first-order condition $0\in z+\partial\psi(w)$ of \eqref{define_wbar}, $y$ optimizes both problems \eqref{define_wbar} with $z=-u$ and $z=-v$, respectively, i.e., $y=\bar w_{\psi}(-u)=\bar w_{\psi}(-v)$. By assumption (ii), $y\in Y$ and $u$ must be equal to $v$, i.e., $\partial\psi(y)$ must be a singleton. Since for arbitrarily chosen $y\in Y$, such arguments hold, by convexity of $\psi$, $\psi$ is differentiable on $Y$. Furthermore, by the first-order condition of \eqref{define_wbar},
    \begin{align}
        z+\nabla\psi(\bar w_{\psi}(z))=0,\quad z\in\R^n.\label{psi_diff}
    \end{align}

    Note that the optimal objective value of the right-hand side of \eqref{define_wbar} is $-\psi^*(-z)$ and $\operatorname{cl}(Y)$ is compact. By assumptions (i) and (ii), and the extended-real-valued extension of Danskin's theorem (\citealt{bertsekas1971control}[Proposition A.22]), 
    $\psi^*$ is differentiable with $\nabla\psi^*(-z)=\bar w_\psi(z)$ for all $z\in\R^n$.
    Moreover, due to assumption (iii), $\nabla\psi^*$ is differentiable everywhere, i.e., $\psi^*$ is twice differentiable on $\R^n$. Since $\bar w_{\psi}(-\cdot)=\nabla\psi^*(\cdot)$,
    strong convexity of $\psi(=\psi^{**})$ follows from assumption (ii) and \cite{rockafellar2009variational}[Proposition 12.60]. Moreover, by assumption (ii) and \eqref{psi_diff}, $\partial\psi=-\bar w_\psi^{-1}:Y\rightarrow\R^n$ is bijective. Following \cite{rockafellar1970convex}[Corollary 26.3.1], $\psi$ is essentially smooth.\qed
\end{proof}

Proposition~\ref{prop:necessary_conditions} demonstrates why the Legendre-type assumptions are necessary. If one wants a regularized optimizer that is uniquely defined for every finite parameter, parameterizes the relative interior of the feasible region through a Lipschitz bijection, and is differentiable everywhere, then the reduced regularizer must be strongly convex, essentially smooth, and have a twice differentiable convex conjugate. This separates our construction from other non-Legendre regularized approaches, for instance, the quadratic smoothing approach for decision-focused combinatorial optimization~\citep{wilder2019melding}, where a small squared-norm penalty is added to a linear optimization problem to obtain a quadratic program. 
Such a quadratic term, together with the hard constraints, does not lead to an admissible Legendre regularizer.
Specifically, when $S$ is full-dimensional and bounded, with
\[
    \phi(w)=\frac{\tau}{2}\|w\|^2_2+\mathbb I_S(w),
    \qquad \tau>0,
\]
the gradient of $\phi$ remains bounded as $w$ approaches the boundary of $S$. Hence $\phi$ is not essentially smooth. In fact, the optimizer is the Euclidean projection of $-z/\tau$ onto $S$, 
so different values of $z$, along directions in the normal cone, may correspond to the same solution lying on the boundary of $S$.
Thus, this map is not a bijection from the latent space onto $S^\circ$. Moreover, for polyhedral $S$, this projection map is only piecewise affine: it is smooth inside a fixed active set, but its Jacobian is piecewise constant and changes discontinuously when the active set changes.


\begin{figure}[tp!]
    \centering
    \includegraphics[width=\linewidth]{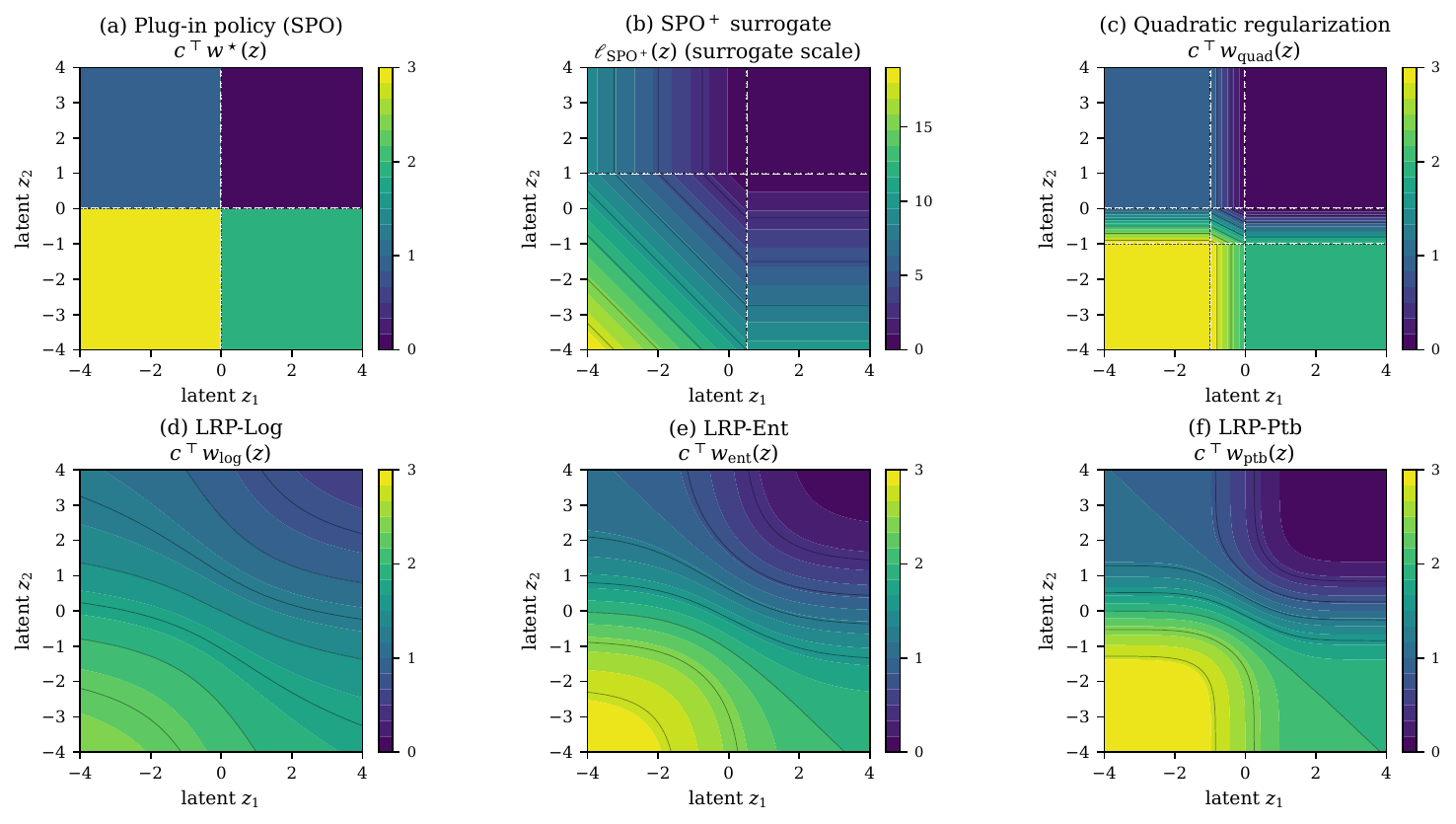}
    \caption{Loss landscapes over $S=[0,1]^2$ with true cost $c=(1,2)$. Panels (a), (c), (d), (e), and (f) show the downstream loss $c^\top w(z)$ induced by the plug-in policy, quadratic regularization, $\mathsf{LRP}$-Log, $\mathsf{LRP}$-Ent, and $\mathsf{LRP}$-Ptb, respectively. Panel (b) shows the SPO${}^+$ surrogate value. Dashed lines mark discontinuities, kink locations, or active-set changes; the three LRP panels in the bottom row have no such finite latent boundary and produce smooth interior policies.}
    \label{fig:toy_loss_landscape}
\end{figure}

Figure~\ref{fig:toy_loss_landscape} illustrates this distinction on a two-dimensional example. 
Let $S=[0,1]^2$ and let the true linear downstream cost be $c^\top w$ with $c=(1,2)$. For a latent vector $z\in\mathbb R^2$, the plug-in policy induces the discontinuous surface $c^\top w^\star(z)$, while the SPO${}^+$ surrogate is continuous but piecewise linear~\citep{elmachtoub2022smart}. Quadratic regularization smoothens the optimizer inside active regions, but the induced downstream loss still has kinks when the solution changes the active set. Specifically, the baselines are given by\begin{align*}
    w^\star(z)&:=\big(\mathds{1}_{(-\infty,0]}(z_i)\big)_{i=1}^2\in\arg\min_{w\in[0,1]^2}\langle z,w\rangle;\\
    \ell_{\textrm{SPO}+}(z)&:=\sum_{i=1}^2\max\{c_i-2z_i,0\}=\max_{w\in[0,1]^2}\langle c-2z,w \rangle;\\
    w_{\mathrm{quad}}(z)&:=\big(\max\big\{\min\{-z_i,1\},0\big\} \big)_{i=1}^2\in\argmin_{w\in[0,1]^2}\Bigl\{\langle z,w\rangle+ \frac{1}{2}\|w\|_2^2 \Bigr\}.
\end{align*} In contrast, the three LRP variants introduced in Section \ref{sec:examples}, $\mathsf{LRP}$-Log, $\mathsf{LRP}$-Ent, and $\mathsf{LRP}$-Ptb, produce smooth interior maps onto this box. Under a unit regularization ($\tau=1$) and a standard Gaussian perturbation, these maps are given by
\[
\begin{aligned}
w_{\log}(z)
&:= \left(\frac{2}{z_i+2+\sqrt{z_i^2+4}}\right)_{i=1}^2
\in \arg\min_{w\in(0,1)^2}\Bigl\{\langle z,w\rangle - \sum_{i=1}^2\log w_i - \sum_{i=1}^2\log(1-w_i) \Bigr\};
\\
w_{\mathrm{ent}}(z)
&:= \left(\frac{1}{1+\exp(z_i)}\right)_{i=1}^2
\in \arg\min_{w\in[0,1]^2} \Bigl\{ \langle z,w\rangle + \sum_{i=1}^2 \left[ w_i\log w_i + (1-w_i)\log(1-w_i) \right] \Bigr\};
\\
w_{\mathrm{ptb}}(z)
&:= \big( \mathbb P \left( z_i+Z_i\leq 0 \right) \big)_{i=1}^2
= \big( \Phi(-z_i) \big)_{i=1}^2
= \mathbb E_{Z} \left[ w_{\mathrm{ptb}}^Z(z) \right],\\
&\qquad\qquad\text{where }
\,
w_{\mathrm{ptb}}^Z(z)=\big(\mathds{1}_{(-\infty,0]}(z_i+Z_i)\big)_{i=1}^2 \in \arg\min_{w\in[0,1]^2} \langle z+Z,w\rangle,
\,
Z\sim\mathcal N(0,I_2).
\end{aligned}
\]
Here $0\log 0:=0$.
Clearly from Figure~\ref{fig:toy_loss_landscape}, LRP-induced downstream loss landscapes $c^\top w_{\log}(z)$, $c^\top w_{\mathrm{ent}}(z)$, and $c^\top w_{\mathrm{ptb}}(z)$ are smoother, and have no finite active-set boundaries in the latent space.

\subsection{Expressive Power of LRPs}
Finally, we show that the Legendre-regularized policy class $\Pi_{F,\phi}(\mathcal{G})$ preserves the universal approximation property of the original model class $\mathcal{G}$, demonstrating the strong expressive power of LRPs.

\begin{assumption}[Universal approximation]
    \label{aspt:univerval_approx}
	Let $\mathcal X\subseteq\mathbb R^{k_1}$ be compact. For every continuous function $h:\mathcal X\to\mathbb R^k$
	and every $\epsilon>0$, there exists $g\in\mathcal G$ such that
	\[
	\sup_{x\in\mathcal X}\|g(x)-h(x)\|_2\le \epsilon.
	\]
\end{assumption}

\begin{proposition}
	Suppose $\phi$ is an admissible Legendre regularizer for $S$, and Assumption \ref{aspt:univerval_approx} holds for some compact $\mathcal X$ and the model class $\mathcal G$. Let $\pi:\mathcal X\to S$
	be any continuous feasible policy. Then for every $\epsilon>0$, there exists an LRP $\hat \pi\in \Pi_{F,\phi}(\mathcal{G})$
	such that
	\[
	\sup_{x\in\mathcal X}
	\|\hat\pi(x)-\pi(x)\|_2
	\le \epsilon.
	\]
\end{proposition}

\begin{proof}{Proof.}
    {Fix $\bar w\in S^\circ$. Define
	\[
	M:=\sup_{x\in\mathcal X}\|\pi(x)-\bar w\|_2.
	\]
	Since $\mathcal X$ is compact and $\pi$ is continuous, $M<+\infty$. Choose $\delta\in(0,1)$ sufficiently small such that $\delta M\leq\epsilon/2$, and define
	\[
	\pi'(x):=(1-\delta)\pi(x)+\delta\bar w.
	\]
	Because $S$ is convex, $\pi(x)\in S$, and $\bar w\in S^\circ$, we have $\pi'(x)\in S^\circ$
	for all $x\in\mathcal X$. Moreover, $\pi':\mathcal X\rightarrow S^\circ$ is continuous and
	\[
	\sup_{x\in\mathcal X}
	\|\pi'(x)-\pi(x)\|_2
	=
	\delta\sup_{x\in\mathcal X}\|\bar w-\pi(x)\|_2
	\leq\epsilon/2.
	\]}
	Define $y(x):=N^\top(\pi'(x)-\bar w)$.
	{Since the columns of $N$ form an orthonormal basis of $\ker(A)$, we have $\bar w+Ny(x)=\pi'(x)\in S^\circ$ and hence $y(x)\in Y$.} Since $\pi'$ is continuous and $\nabla\psi$ is continuous on $Y$, the map
	\[
	h(x)
	:=
	-B^\top(BB^\top)^{-1}\nabla\psi(y(x))
	\]
	is continuous. Moreover, $Bh(x)=-\nabla\psi(y(x))$.
	Therefore, by Theorem \ref{thm:main} and property 3 in Lemma \ref{lem:basic}, we have
	\[
	w_{F,\phi}(h(x))=\bar w+N\nabla\psi^*(-Bh(x))=\bar w+Ny(x)=\pi'(x).
	\]
	
	By the universal approximation assumption, choose $g\in\mathcal G$ such that
	\[
	\sup_{x\in\mathcal X}
	\|g(x)-h(x)\|_2
	\le
	\frac{\mu\epsilon}{2\|B\|_2}.
	\]
	Define $\hat\pi(x):=w_{F,\phi}(g(x))\in \Pi_{F,\phi}(\mathcal{G})$.
	Following the Lipschitz continuity of $w_{F,\phi}$ from Theorem \ref{thm:main},
	\[
	\begin{aligned}
		\sup_{x\in\mathcal X}
		\|\hat\pi(x)-\pi(x)\|_2
		&\le
		\sup_{x\in\mathcal X}
		\|w_{F,\phi}(g(x))-w_{F,\phi}(h(x))\|_2
		+
		\sup_{x\in\mathcal X}
		\|\pi'(x)-\pi(x)\|_2  \\
		&\le
		\frac{\|B\|_2}{\mu}
		\sup_{x\in\mathcal X}
		\|g(x)-h(x)\|_2
		+
		\epsilon/2  \\
		&\le
		\epsilon.
	\end{aligned}
	\]\qed
\end{proof}

\section{Numerical Experiments}
\label{sec:numerical_experiments}
We evaluate the LRP framework on three contextual stochastic optimization problems: a single-product newsvendor problem, a two-stage resource allocation problem with box constraints, and a two-stage resource allocation problem with group-budget constraints. 
The newsvendor design is adapted from the numerical experiments of~\citet{zhang2026data}, and the Piecewise Affine Decision Rule (PADR) benchmark is based on the decision-rule method proposed therein. The resource allocation design and the residual-based Sample Average Approximation (SAA) benchmarks are adapted from the Data-Driven SAA (DD-SAA) framework of \citet{kannan2025data}. We compare these benchmarks with three LRP variants defined in Section \ref{sec:examples}: $\mathsf{LRP}$-Log, $\mathsf{LRP}$-Ent, and $\mathsf{LRP}$-Ptb.

\subsection{Experimental Protocol}
\label{subsec:experiment_protocol}

\paragraph{Data splitting and performance evaluation.}
For each problem setting, consistent with recent machine-learning evaluations that select configurations using held-out validation data and assess the selected configuration on a separate test set over repeated runs~\citep{hazimeh2023mind,marx2023calibration}, we first generate five independent tuning replications for hyperparameter selection and, after selection, 50 fresh independent confirmation replications for final reporting. Each replication consists of an i.i.d.\ data set $\mathcal D=\{(x_i,\xi_i)\}_{i=1}^{|\mathcal D|}$ partitioned in proportions $0.5/0.25/0.25$. In every replication, the training subset $\mathcal D_{\mathrm{tr}}$ is used to estimate model parameters, and the validation subset $\mathcal D_{\mathrm{val}}$ is used for method-specific within-training choices, such as early stopping and checkpoint selection, when applicable. In a tuning replication, the third subset is denoted by $\mathcal D_{\mathrm{tune}}$ and is used to score candidate configurations; in a confirmation replication, it is denoted by $\mathcal D_{\mathrm{test}}$ and is reserved for final performance assessment.

Since not all tested methods guarantee feasibility, to ensure the final solution is feasible, for every method $a$, let $\widetilde\pi_a(x)$ denote its raw decision and let $\widehat\pi_a(x)$ denote the restored decision used for evaluation. 
Specifically, with $S_{\mathrm{ra}}$ denoting the corresponding first-stage feasible set,
\[
\widehat\pi_a(x)
:=
\operatorname*{argmin}_{w\in S_{\mathrm{ra}}}
\frac{1}{2}\left\|w-\widetilde\pi_a(x)\right\|_2^2,
\]
and both tuning-set and test-set costs are evaluated using $\widehat\pi_a(x)$.


\paragraph{Hyperparameter tuning.}
For each problem setting, all methods use the same outer tuning protocol. Each candidate configuration is evaluated over the five tuning replications, and the candidate with the lowest mean tuning-set cost is selected. The corresponding search domains and finite candidate sets are reported in Appendix~\ref{subsec:appendix_tuning}.

For $\mathsf{LRP}$-Log, $\mathsf{LRP}$-Ent, $\mathsf{LRP}$-Ptb, and PADR, hyperparameter tuning uses the same sequential Gaussian-process Bayesian-optimization procedure with expected improvement and a common prespecified total budget of candidate evaluations within each problem setting~\citep{jones1998efficient,snoek2012practical}. For the residual-based SAA benchmarks, the active tuning choices form small finite sets and are exhaustively enumerated using the same tuning score. After selection, each configuration is held fixed throughout the 50 confirmation replications.

\paragraph{LRP variants.}
The three tested LRP methods use the same prediction architecture for learning the latent model $g_\theta:\mathcal X\to\mathbb R^k$ and differ only in the regularizer used in the optimization layer. Specifically, $g_\theta$ is a multilayer perceptron with three fully connected hidden layers of widths 16, 32, and 16, a ReLU activation after each hidden layer, and a linear output layer of width $k$. All networks are trained with Adam using mini-batches of size $50$ and validation patience $100$, for at most $10{,}000$ epochs in the newsvendor experiment and $2{,}000$ epochs in the resource allocation experiments.
For $\phi\in\{\phi_{\log},\phi_{\mathrm{ent}},\phi_{\mathrm{ptb}}\}$, the corresponding LRP is
\[
    \pi_{\theta,\phi}(x)
    :=
    w_{F,\phi}\bigl(g_\theta(x)\bigr).
\]
The parameter $\tau$ retains the meaning introduced in Section~\ref{sec:examples}: it scales the explicit regularizers $\phi_{\log}$ and $\phi_{\mathrm{ent}}$, while for $\phi_{\mathrm{ptb}}$ it is the standard deviation of the Gaussian perturbation $Z$. During training, the smoothing parameter at epoch $e$ is $\tau_e=\max\{\tau_{\min},\tau_0\delta^{\lfloor e/s\rfloor}\}$, where $\tau_0$, $\tau_{\min}$, the decay factor $\delta$, and the update interval $s$ are tuned jointly with the remaining method-specific hyperparameters; the value associated with the validation-selected checkpoint is used for evaluation. The predictor is trained by minimizing empirical decision regret:
\[
    \min_\theta
    \frac{1}{|\mathcal D_{\mathrm{tr}}|}
    \sum_{(x_i,\xi_i)\in\mathcal D_{\mathrm{tr}}}
    \left[
    \ell\left(\pi_{\theta,\phi}(x_i);\xi_i\right)
    -
    \min_{w\in S}\ell(w;\xi_i)
    \right].
\]
In both the newsvendor and resource allocation experiments, the outer loss is convex and piecewise linear in the decision. At non-differentiable points of the loss function, we select a valid Clarke subgradient with respect to the decision and backpropagate it through the Jacobian of the regularized optimization map.

\paragraph{Benchmark methods.}
We compare against recent representative baselines rather than reimplementing the full set of earlier contextual optimization methods. On the direct policy-learning side, we use PADR~\citep{zhang2026data}, a recent state-of-the-art decision-rule method that has been shown to outperform several prominent decision-rule approaches, including neural-network and RKHS-based decision rules. PADR parameterizes each decision coordinate as the difference of two max-affine functions and fits the resulting decision rule by empirical risk minimization.
On the predict-then-optimize side, we consider the residual-based SAA family of~\citet{kannan2025data}: Empirical Residual SAA (ER-SAA) constructs scenarios using residuals from the model fitted on the training samples, whereas Jackknife SAA (J-SAA) uses leave-one-out residuals. Jackknife-plus SAA (J${}^+$-SAA) further pairs each leave-one-out residual with its corresponding leave-one-out prediction, following the extended preprint~\citep{kannan2022data_extended}. Further implementation details for these three methods are presented in Appendix~\ref{subsec:appendix_benchmark}.


We do not include the standard SPO${}^+$ surrogate~\citep{elmachtoub2022smart} as a separate benchmark because it is designed for problems in which the predicted uncertainty enters as coefficients of a linear downstream objective. In our experiments, demand instead enters through inventory imbalance or recourse constraints, producing general convex, piecewise-linear losses.

\subsection{Single-Product Newsvendor Problem}
\label{subsec:experiment_newsvendor}

The first experiment considers a single-product newsvendor problem with contextual demand. Given context $x$, the decision $w$ is an order quantity and the uncertain parameter $\xi$ is the realized demand. The downstream loss is
\[
    \ell_{\mathrm{nv}}(w;\xi)
    =
    c_b(\xi-w)_+
    +
    c_h(w-\xi)_+,
    \qquad
    (c_b,c_h)=(8,2).
\]
Although the classical newsvendor formulation without side constraints imposes only $w\geq 0$, the LRP layers are implemented on $S_{\mathrm{nv}}=[0,w^{\max}]$ with $w^{\max}=100$ so that each regularized optimization map is well-defined for every latent input. For all methods, any negative raw order quantity is restored to zero before the unconstrained newsvendor loss is evaluated.

We use the sparse max-affine newsvendor design of~\citet{zhang2026data} and modify it by applying a sign-preserving power transformation to its two informative context coordinates to introduce more nonlinearity. For each observation, $x=(x_1,\ldots,x_{k_1})$ has independent coordinates $x_j\sim\operatorname{Unif}[-1,1]$. The conditional demand depends only on the first two coordinates, while the remaining $k_1-2$ coordinates are irrelevant noise features. For $\gamma>0$, define
\[
    \operatorname{spow}_{\gamma}(t)
    :=
    \operatorname{sign}(t)|t|^\gamma,
    \qquad
    \widetilde x_j
    :=
    \operatorname{spow}_{\gamma}(x_j),
    \quad j=1,2.
\]
Conditional demand is generated as
\[
    \xi
    =
    f_{\mathrm{nv},\gamma}(x)
    +
    \sigma\zeta,
    \qquad
    \zeta\sim\mathcal N(0,1),
    \qquad
    \zeta\perp x,
\]
where
\[
    f_{\mathrm{nv},\gamma}(x)
    =
    10
    +
    \max\left\{
    5\widetilde x_1-10\widetilde x_2,
    -10\widetilde x_1+5\widetilde x_2,
    15\widetilde x_1
    \right\}.
\]
We vary $\gamma\in\{1,3,5\}$ and $\sigma\in\{1,3\}$ to control the nonlinearity and noise standard deviation, respectively. Note that under the data-generating process, $f_{\mathrm{nv},\gamma}(x)\leq25$, and the largest noise standard deviation is $3$. Since $c_b/(c_b+c_h)=0.8$, the conditional population-optimal order quantity is at most $25+3\Phi^{-1}(0.8)\approx27.53$. Thus, the maximum order quantity $w^{\max}=100$ is a deliberately loose upper bound that provides a compact domain for the LRP maps without altering the population-optimal newsvendor policy. Under this box constraint, all three LRP maps admit closed-form expressions and can therefore be evaluated efficiently; their explicit formulas are provided in Appendix~\ref{subsec:appendix_cf}.

For each confirmation replication, the newsvendor test cost is
\[
    \widehat{\operatorname{Cost}}_{\mathrm{nv}}(a)
    =
    \frac{1}{|\mathcal D_{\mathrm{test}}|}
    \sum_{(x_i,\xi_i)\in\mathcal D_{\mathrm{test}}}
    \left[
    c_b\bigl(\xi_i-\widehat\pi_a(x_i)\bigr)_+
    +
    c_h\bigl(\widehat\pi_a(x_i)-\xi_i\bigr)_+
    \right].
\]

\begin{table}[tp!]
\centering
\caption{Out-of-sample costs for the single-product newsvendor experiment. {Here, $|\mathcal D|$, $k_1$, $\gamma$, and $\sigma$ denote the total sample size, context dimension, nonlinearity degree, and noise standard deviation, respectively.} Entries are mean $\pm$ standard deviation over 50 independent confirmation replications; lower values are better.}
\label{tab:newsvendor_unconstrained_confirmation50}
\resizebox{\textwidth}{!}{%
\begin{tabular}{lccccccc}
\toprule
Setting $(|\mathcal D|,k_1,\gamma,\sigma)$ & LRP-Log & LRP-Ent & LRP-Ptb & PADR & ER-SAA & J-SAA & J$^+$-SAA \\
\midrule
$(1000,20,3,1)$ & $\mathbf{\phantom{0}4.611 \pm 0.506}$ & $\phantom{0}4.731 \pm 0.560$ & $\phantom{0}4.758 \pm 0.535$ & $\phantom{0}6.719 \pm 0.981$ & $11.862 \pm 0.671$ & $11.858 \pm 0.648$ & $11.858 \pm 0.647$ \\
$(\phantom{0}500,20,3,1)$ & $\mathbf{\phantom{0}6.640 \pm 0.995}$ & $\phantom{0}6.995 \pm 1.256$ & $\phantom{0}6.823 \pm 1.084$ & $10.500 \pm 1.619$ & $12.105 \pm 0.804$ & $12.079 \pm 0.762$ & $12.082 \pm 0.768$ \\
$(2000,20,3,1)$ & $\mathbf{\phantom{0}3.612 \pm 0.196}$ & $\phantom{0}3.737 \pm 0.224$ & $\phantom{0}3.765 \pm 0.287$ & $\phantom{0}5.538 \pm 0.650$ & $11.604 \pm 0.428$ & $11.601 \pm 0.420$ & $11.601 \pm 0.420$ \\
$(1000,\phantom{0}6,3,1)$ & $\mathbf{\phantom{0}3.332 \pm 0.208}$ & $\phantom{0}3.435 \pm 0.237$ & $\phantom{0}3.488 \pm 0.284$ & $\phantom{0}3.808 \pm 0.252$ & $11.270 \pm 0.728$ & $11.625 \pm 0.634$ & $11.625 \pm 0.634$ \\
$(1000,20,1,1)$ & $\phantom{0}3.927 \pm 0.261$ & $\mathbf{\phantom{0}3.789 \pm 0.238}$ & $\phantom{0}3.893 \pm 1.048$ & $\phantom{0}3.996 \pm 0.630$ & $10.074 \pm 0.529$ & $10.071 \pm 0.516$ & $10.072 \pm 0.517$ \\
$(1000,20,5,1)$ & $\phantom{0}6.019 \pm 1.475$ & $\phantom{0}6.007 \pm 1.391$ & $\mathbf{\phantom{0}5.638 \pm 0.557}$ & $\phantom{0}7.897 \pm 1.257$ & $11.332 \pm 0.873$ & $11.600 \pm 0.759$ & $11.599 \pm 0.761$ \\
$(1000,20,3,3)$ & $\mathbf{10.921 \pm 0.758}$ & $10.996 \pm 0.884$ & $11.044 \pm 0.923$ & $11.926 \pm 0.901$ & $14.067 \pm 0.846$ & $14.070 \pm 0.827$ & $14.070 \pm 0.830$ \\
\bottomrule
\end{tabular}%
}
\end{table}

\begin{figure}[tp!]
    \centering
    \includegraphics[width=\linewidth]{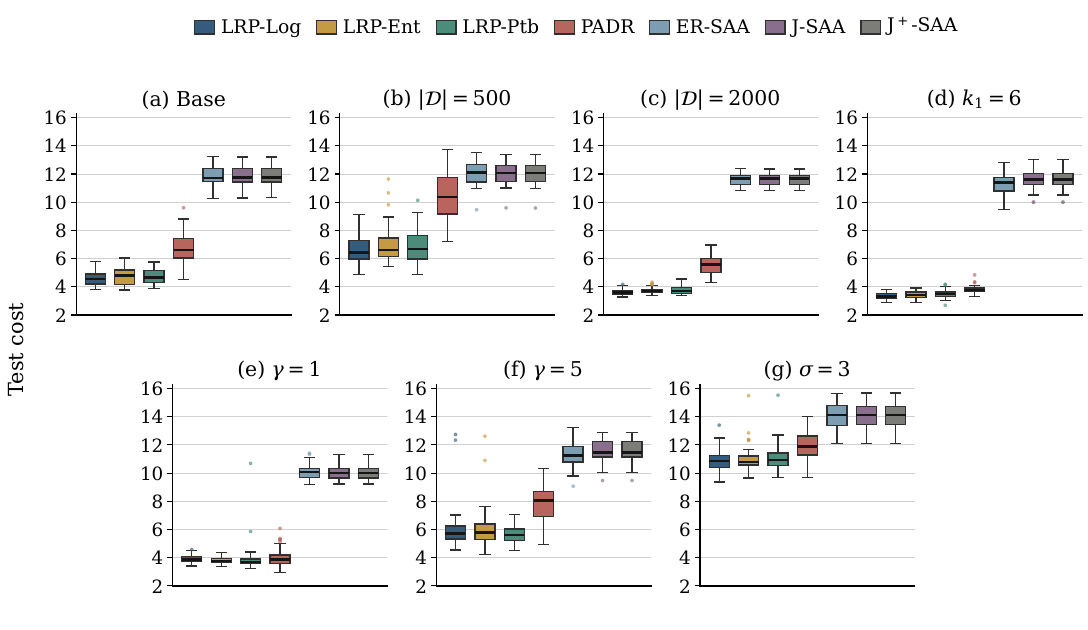}
    \caption{Out-of-sample cost distributions for the single-product newsvendor experiment. Each box summarizes test costs over 50 independent confirmation replications. Panel (a) shows the baseline setting $(|\mathcal D|,k_1,\gamma,\sigma)=(1000,20,3,1)$; panels (b)--(g) change only the parameter indicated in the panel title while holding the remaining parameters at their baseline values. Within each panel, methods follow the legend order from left to right. The center line is the median, the box spans the interquartile range, the whiskers extend to the most extreme observations within 1.5 interquartile ranges of the lower and upper quartiles, and points beyond the whiskers are shown individually; lower values are better.}
    \label{fig:newsvendor_unconstrained_boxplots}
\end{figure}

Table~\ref{tab:newsvendor_unconstrained_confirmation50} reports the test costs across the 50 confirmation replications for the newsvendor experiment under the specified values of $|\mathcal D|$, $k_1$, $\gamma$, and $\sigma$, while Figure~\ref{fig:newsvendor_unconstrained_boxplots} visualizes the distributions underlying these summary statistics. The seven panels use a common vertical scale: panel (a) is the baseline setting $(|\mathcal D|,k_1,\gamma,\sigma)=(1000,20,3,1)$, and each remaining panel changes only the parameter indicated in its title. 
Across all settings, the LRP variants substantially outperform the residual-based SAA methods and generally outperform PADR. Their advantage over PADR is particularly pronounced when the context dimension or the nonlinearity of the demand function increases. For example, the best LRP reduces the mean cost relative to PADR by approximately $12.5\%$ when $k_1=6$, compared with $31.4\%$ when $k_1=20$ in the baseline setting. Similarly, this reduction increases from approximately $5.2\%$ when $\gamma=1$ to $31.4\%$ and $28.6\%$ when $\gamma=3$ and $\gamma=5$, respectively. These findings are consistent with the expressive-power result in Section~\ref{sec:legendre_regularized_policies}: when combined with a flexible latent model, the LRP construction can represent rich nonlinear context-decision relationships while maintaining feasibility. The results also show that the LRP methods benefit substantially from larger sample sizes. Finally, although larger noise increases the costs of all methods, the LRP variants retain their performance advantage when $\sigma=3$.

\subsection{Resource Allocation with Box Constraints}
\label{subsec:experiment_resource_box}

The second experiment considers the two-stage resource allocation model used in the DD-SAA computational study of \citet{kannan2025data}, augmented with the first-stage box constraint $0\le w\le w^{\max}$, where $w^{\max}=100 \cdot \mathbf 1_n\in\mathbb R_{++}^n$. Let $\mathcal I$ denote the resource set, with $|\mathcal I|=n$, and let $\mathcal J$ denote the demand-class set. For a first-stage decision $w$ and realized uncertain parameter $\xi$, define
\[
    \ell_{\mathrm{ra}}(w;\xi)
    :=
    \bigl(c^{\mathrm{ra}}\bigr)^\top w
    +
    \mathcal R(w,\xi),
\]
where
\[
\begin{aligned}\mathcal R(w,\xi)
    =
    \min_{v,u}\left\{q^\top u:\sum_{j\in\mathcal J}v_{ij}\le \rho_i w_i,i\in\mathcal I,~\sum_{i\in\mathcal I}\kappa_{ij}v_{ij}+u_j\ge \xi_j,j\in\mathcal J,~v\ge 0, u\ge 0\right\}.
\end{aligned}
\]
The first-stage feasible set is
\[
    S_{\mathrm{ra}}^{\mathrm{box}}
    =
    \left\{
    w\in\mathbb R^n:
    0\le w\le w^{\max}
    \right\}.
\]

We use the same 20-resource, 30-demand-class base instance as in the DD-SAA computational study of~\citet{kannan2025data}, retaining its first-stage cost vector $c^{\mathrm{ra}}$, yield coefficients $\rho$, service-rate coefficients $\kappa$, recourse cost vector $q$, and demand-model parameter-generation procedure. Contexts are generated from the sparse folded-normal model of~\citet{kannan2025data}. Specifically, $x$ follows the multivariate folded-normal distribution used therein, only its first three coordinates are informative, and conditional demand is generated as
\[
    \xi_j
    = \alpha_j + \sum_{\ell=1}^{3} \beta_{j\ell}x_\ell^\gamma + \sigma\zeta_j,
    \qquad
    \zeta_j\stackrel{\mathrm{iid}}{\sim}\mathcal N(0,1),
    \qquad
    \zeta\perp x,
    \quad j\in\mathcal J.
\]
In each replication, the parameters $\alpha_j$, $\beta_{j\ell}$, and the folded-normal covariance matrix are generated independently according to the original procedure and then held fixed for all observations in that replication; the same generated instance and data split are used for all methods within that replication. The exponent $\gamma$ corresponds to the model-degree parameter denoted by $p$ in~\citet{kannan2025data}; we use $\gamma\in\{1,3,5\}$, corresponding to linear, cubic, and quintic dependence on the informative context coordinates. The parameter $\sigma\in\{5,15\}$ is the conditional noise standard deviation. We use the homoscedastic specification, so $\operatorname{Var}(\xi_j\mid x)=\sigma^2$ for every $j\in\mathcal J$. {Because the feasible region and regularizers are separable across resources, all three LRP maps admit coordinatewise closed-form expressions and can be evaluated efficiently. The explicit formulas are provided in Appendix~\ref{subsec:appendix_cf}.}


For each confirmation replication, the box-constrained resource allocation test cost is
\[
    \widehat{\operatorname{Cost}}_{\mathrm{ra}}^{\mathrm{box}}(a)
    =
    \frac{1}{|\mathcal D_{\mathrm{test}}|}
    \sum_{(x_i,\xi_i)\in\mathcal D_{\mathrm{test}}}
    \left[
    \bigl(c^{\mathrm{ra}}\bigr)^\top\widehat\pi_a(x_i)
    +
    \mathcal R\bigl(\widehat\pi_a(x_i),\xi_i\bigr)
    \right].
\]

Table~\ref{tab:resource_allocation_box_confirmation50} reports the test costs across the 50 confirmation replications for the resource allocation experiment with box constraints. A complementary distributional view of these results is provided in Figure~\ref{fig:resource_allocation_box_boxplots} in Appendix~\ref{sec:appendix_experiment}. Across six of the seven settings, an LRP variant achieves the lowest mean test cost. The advantage over the benchmark methods is particularly evident in settings with higher-dimensional contexts and nonlinear dependence.

\begin{table}[tp!]
\centering
\caption{Out-of-sample costs for resource allocation with box constraints. {Here, $|\mathcal D|$, $k_1$, $\gamma$, and $\sigma$ denote the total sample size, context dimension, nonlinearity degree, and noise standard deviation, respectively.} Entries are mean $\pm$ standard deviation over 50 independent confirmation replications; lower values are better.}
\label{tab:resource_allocation_box_confirmation50}
\resizebox{\textwidth}{!}{%
\begin{tabular}{lccccccc}
\toprule
Setting $(|\mathcal D|,k_1,\gamma,\sigma)$ & LRP-Log & LRP-Ent & LRP-Ptb & PADR & ER-SAA & J-SAA & J$^+$-SAA \\
\midrule
$(200,\phantom{0}3,3,\phantom{0}5)$ & $1387.6 \pm \phantom{0}237.1$ & $\mathbf{1378.1 \pm \phantom{0}235.2}$ & $1378.2 \pm \phantom{0}236.5$ & $1526.1 \pm \phantom{0}242.8$ & $1559.0 \pm \phantom{0}239.4$ & $1630.8 \pm \phantom{0}243.0$ & $1631.4 \pm \phantom{0}243.2$ \\
$(100,\phantom{0}3,3,\phantom{0}5)$ & $1493.0 \pm \phantom{0}360.7$ & $\mathbf{1438.0 \pm \phantom{0}357.6}$ & $1444.5 \pm \phantom{0}356.3$ & $1877.6 \pm \phantom{0}467.2$ & $1663.1 \pm \phantom{0}357.4$ & $1662.4 \pm \phantom{0}356.3$ & $1663.4 \pm \phantom{0}356.7$ \\
$(500,\phantom{0}3,3,\phantom{0}5)$ & $1417.9 \pm \phantom{0}136.7$ & $1490.6 \pm \phantom{0}145.0$ & $\mathbf{1403.4 \pm \phantom{0}140.4}$ & $1916.3 \pm \phantom{0}256.8$ & $1586.0 \pm \phantom{0}133.1$ & $1660.5 \pm \phantom{0}137.8$ & $1660.7 \pm \phantom{0}137.9$ \\
$(200,10,3,\phantom{0}5)$ & $1493.9 \pm \phantom{0}241.7$ & $\mathbf{1451.5 \pm \phantom{0}236.4}$ & $1983.3 \pm \phantom{0}293.5$ & $1842.1 \pm \phantom{0}289.2$ & $1595.8 \pm \phantom{0}240.7$ & $1653.3 \pm \phantom{0}244.7$ & $1654.1 \pm \phantom{0}245.2$ \\
$(200,\phantom{0}3,1,\phantom{0}5)$ & $\phantom{0}964.1 \pm \phantom{00}24.5$ & $\phantom{0}952.2 \pm \phantom{00}22.9$ & $\phantom{0}982.3 \pm \phantom{00}36.5$ & $1009.5 \pm \phantom{00}42.9$ & $\mathbf{\phantom{0}947.1 \pm \phantom{00}22.7}$ & $\mathbf{\phantom{0}947.1 \pm \phantom{00}22.7}$ & $\mathbf{\phantom{0}947.1 \pm \phantom{00}22.7}$ \\
$(200,\phantom{0}3,5,\phantom{0}5)$ & $5856.8 \pm 2773.4$ & $\mathbf{5822.5 \pm 2777.5}$ & $5888.4 \pm 2779.6$ & $6258.5 \pm 2783.1$ & $6465.3 \pm 2754.8$ & $6457.9 \pm 2754.3$ & $6462.2 \pm 2754.2$ \\
$(200,\phantom{0}3,3,15)$ & $1421.3 \pm \phantom{0}239.6$ & $1414.2 \pm \phantom{0}236.3$ & $\mathbf{1410.2 \pm \phantom{0}233.7}$ & $1548.4 \pm \phantom{0}247.8$ & $1564.7 \pm \phantom{0}238.8$ & $1633.7 \pm \phantom{0}244.2$ & $1634.3 \pm \phantom{0}244.4$ \\
\bottomrule
\end{tabular}%
}
\end{table}

\subsection{Resource Allocation with Group-Budget Constraints}
\label{subsec:experiment_resource_general}

To test more general feasible regions, the third experiment uses the same two-stage recourse model and demand generator as Section~\ref{subsec:experiment_resource_box}, but modifies the box-only first-stage feasible region by adding two group-budget constraints. With $|\mathcal I|=n=20$, define
\[
    \bar s_1
    =
    0.40\sum_{i=1}^{20}w_i^{\max},
    \qquad
    \bar s_2
    =
    0.45\sum_{i=1}^{20}w_i^{\max},
\]
and set
\[
    S_{\mathrm{ra}}^{\mathrm{gb}}
    =
    \left\{
    w\in\mathbb R^n:
    0\le w_i\le w_i^{\max}\ \forall i,\quad
    \sum_{i=1}^{10}w_i\le \bar s_1,\quad
    \sum_{i=11}^{20}w_i\le \bar s_2
    \right\}.
\]
With $w_i^{\max}=100$ for all $i$ and ten resources in each group, each group has aggregate box capacity $1000$; hence $\bar s_1=800$ and $\bar s_2=900$ equal $80\%$ and $90\%$ of their respective group capacities. This experiment assesses the methods under nonseparable linear first-stage constraints. The downstream loss and test-cost definition are the same as in Section~\ref{subsec:experiment_resource_box}; only the first-stage feasible set changes.

All three LRP variants use latent dimension $k=n$ and produce decisions in the original feasible region $S_{\mathrm{ra}}^{\mathrm{gb}}$. Details on evaluating and differentiating the corresponding optimization layers are provided in Appendix~\ref{subsec:appendix_cf}.

Table~\ref{tab:resource_allocation_general_confirmation50} presents the out-of-sample costs from 50 independent confirmation replications under group-budget constraints. The corresponding cost distributions are shown in Figure~\ref{fig:resource_allocation_group_budget_boxplots} in Appendix~\ref{sec:appendix_experiment}. Across six of the seven settings, $\mathsf{LRP}$-Ent achieves the lowest mean test cost among all methods. The only exception is the linear setting with $\gamma=1$, where the residual-based SAA methods perform slightly better. Overall, these results show that the LRP framework remains effective when the feasible region contains nonseparable group-budget constraints, with its advantage being particularly evident under complex context-demand relationships.

\begin{table}[tp!]
\centering
\caption{Out-of-sample costs for resource allocation with group-budget constraints. {Here, $|\mathcal D|$, $k_1$, $\gamma$, and $\sigma$ denote the total sample size, context dimension, nonlinearity degree, and noise standard deviation, respectively.} Entries are mean $\pm$ standard deviation over 50 independent confirmation replications; lower values are better.}
\label{tab:resource_allocation_general_confirmation50}
\resizebox{\textwidth}{!}{%
\begin{tabular}{lccccccc}
\toprule
Setting $(|\mathcal D|,k_1,\gamma,\sigma)$ & LRP-Log & LRP-Ent & LRP-Ptb & PADR & ER-SAA & J-SAA & J$^+$-SAA \\
\midrule
$(200,\phantom{0}3,3,\phantom{0}5)$ & $1499.0 \pm \phantom{0}262.4$ & $\mathbf{1458.2 \pm \phantom{0}262.1}$ & $1560.9 \pm \phantom{0}262.0$ & $1897.2 \pm \phantom{0}352.2$ & $1606.5 \pm \phantom{0}263.1$ & $1674.5 \pm \phantom{0}264.3$ & $1675.2 \pm \phantom{0}264.4$ \\
$(100,\phantom{0}3,3,\phantom{0}5)$ & $1543.5 \pm \phantom{0}399.8$ & $\mathbf{1518.0 \pm \phantom{0}392.6}$ & $1600.1 \pm \phantom{0}401.7$ & $1880.7 \pm \phantom{0}437.0$ & $1710.6 \pm \phantom{0}389.4$ & $1709.4 \pm \phantom{0}388.1$ & $1710.5 \pm \phantom{0}388.4$ \\
$(500,\phantom{0}3,3,\phantom{0}5)$ & $1487.8 \pm \phantom{0}151.6$ & $\mathbf{1474.0 \pm \phantom{0}151.2}$ & $1711.8 \pm \phantom{0}192.7$ & $1882.1 \pm \phantom{0}213.0$ & $1710.6 \pm \phantom{0}147.3$ & $1710.4 \pm \phantom{0}147.1$ & $1710.6 \pm \phantom{0}147.2$ \\
$(200,10,3,\phantom{0}5)$ & $1551.3 \pm \phantom{0}266.3$ & $\mathbf{1516.4 \pm \phantom{0}262.1}$ & $1567.7 \pm \phantom{0}265.5$ & $1835.0 \pm \phantom{0}293.5$ & $1696.0 \pm \phantom{0}266.5$ & $1694.6 \pm \phantom{0}264.5$ & $1695.5 \pm \phantom{0}265.1$ \\
$(200,\phantom{0}3,1,\phantom{0}5)$ & $\phantom{0}964.8 \pm \phantom{00}23.9$ & $\phantom{0}951.0 \pm \phantom{00}22.8$ & $\phantom{0}981.5 \pm \phantom{00}27.9$ & $1017.5 \pm \phantom{00}53.2$ & $\mathbf{\phantom{0}947.1 \pm \phantom{00}22.7}$ & $\mathbf{\phantom{0}947.1 \pm \phantom{00}22.7}$ & $\mathbf{\phantom{0}947.1 \pm \phantom{00}22.7}$ \\
$(200,\phantom{0}3,5,\phantom{0}5)$ & $6050.3 \pm 2809.6$ & $\mathbf{6032.3 \pm 2812.3}$ & $6381.6 \pm 2811.8$ & $6379.8 \pm 2808.9$ & $6563.0 \pm 2797.1$ & $6553.7 \pm 2796.6$ & $6558.8 \pm 2796.5$ \\
$(200,\phantom{0}3,3,15)$ & $1509.7 \pm \phantom{0}261.1$ & $\mathbf{1486.0 \pm \phantom{0}265.1}$ & $1531.0 \pm \phantom{0}262.7$ & $1569.5 \pm \phantom{0}258.6$ & $1677.7 \pm \phantom{0}265.7$ & $1677.4 \pm \phantom{0}265.4$ & $1678.1 \pm \phantom{0}265.5$ \\
\bottomrule
\end{tabular}%
}
\end{table}

\section{Concluding Remarks}
This paper introduces Legendre-regularized policies for learning contextual decisions under hard constraints. Our results demonstrate that an appropriate regularizer can simultaneously enforce hard constraints, provide stable gradients, and preserve the expressive power of the underlying prediction model. Several directions merit further investigation. One direction is to develop scalable algorithms for evaluating and differentiating Legendre-regularized optimization layers in large-scale problems. Another is to study how the choice of the regularizer and smoothing parameter affects statistical performance and computational efficiency. Extending the framework to settings with nonconvex or discrete feasible regions is also an important direction.



\bibliographystyle{plainnat}
\bibliography{reference}

\clearpage
\begin{APPENDICES}
\section{Additional Experimental Details and Results}\label{sec:appendix_experiment}

\subsection{Computational Environment.}\label{subsec:appendix_computation}

All experiments were implemented in Python~3.11 using PyTorch~2.7.1 and run without GPU acceleration on a server equipped with two Intel Xeon Platinum 8575C processors and approximately 378~GiB of usable memory. The PADR optimization subproblems were solved using Gurobi~12.0.3, whereas the recourse and deterministic-equivalent SAA problems in the resource allocation experiments were solved using HiGHS~1.8.0 through \texttt{scipy.optimize.linprog}.

\subsection{Hyperparameter Search Domains.}\label{subsec:appendix_tuning}

Table~\ref{tab:hyperparameter_domains} reports the search domains used by the tuning procedure. For $\mathsf{LRP}$-Ptb, $\tau_0$ and $\tau_{\min}$ denote the initial and minimum Gaussian perturbation standard deviations, respectively.

\begin{table}[ht!]
\centering
\small
\caption{Hyperparameter search domains used in the numerical experiments. Intervals are continuous unless a parameter is explicitly identified as integer-valued.}
\label{tab:hyperparameter_domains}
\renewcommand{\arraystretch}{1.05}
\begin{tabular}{@{}>{\raggedright\arraybackslash}p{0.20\textwidth}@{\hspace{0.02\textwidth}}>{\raggedright\arraybackslash}p{0.78\textwidth}@{}}
\toprule
Method and experiment & Search domain \\
\midrule
$\mathsf{LRP}$ variants: newsvendor & \emph{Smoothing levels:} $\tau_0\in[10^{-2},50]$, $\tau_{\min}\in[10^{-6},\tau_0]$.\newline \emph{Smoothing schedule:} $\delta\in[0.5,1]$, integer update interval $s\in\{1,\ldots,1000\}$.\newline \emph{Optimization:} learning rate in $[10^{-3},1]$, weight decay in $[0,2]$.\newline \emph{Perturbation computation:} $\mathsf{LRP}$-Ptb is evaluated using the exact Gaussian expectation. \\
\addlinespace[3pt]
$\mathsf{LRP}$ variants: resource allocation & \emph{Smoothing levels:} $\tau_0\in[10^{-2},10]$, $\tau_{\min}\in[10^{-6},\tau_0]$.\newline \emph{Smoothing schedule:} $\delta\in[0.5,1]$, integer update interval $s\in\{1,\ldots,1000\}$.\newline \emph{Optimization:} learning rate in $[10^{-4},10^{-1}]$, weight decay in $[0,2]$.\newline \emph{Perturbation computation:} Under box constraints, $\mathsf{LRP}$-Ptb is evaluated using the exact Gaussian expectation. Under group-budget constraints, the expectation is estimated using 50 independently drawn Gaussian perturbations per input, resampled at each layer evaluation. \\
\addlinespace[3pt]
PADR: all experiments & \emph{Policy complexity:} $(K_1,K_2)\in\{0,\ldots,40\}^2\setminus\{(0,0)\}$.\newline \emph{Active-piece control:} $\epsilon,\epsilon_{\mathrm{sh}}\in\{0,\ldots,3000\}$, shrinking ratio in $[0,1]$.\newline \emph{ESMM run control:} number of random starts in $\{1,\ldots,10\}$, number of iterations in $\{1,\ldots,100\}$.\newline \emph{Sampling schedule:} $\alpha\in\{1,\ldots,500\}$, $\beta,N_0\in\{0,\ldots,20\}$.\newline \emph{Coefficient control:} coefficient bound $\mu\in[10^{-3},100]$, proximal weight $\eta\in[0,1]$.\newline \emph{Subsampling:} sampling indicator in $\{\mathrm{false},\mathrm{true}\}$. \\
\addlinespace[3pt]
Residual-based SAA & \emph{Regression model:} ordinary least squares for ER-SAA, J-SAA, and J${}^+$-SAA.\newline \emph{Conditional scale:} ER-SAA enumerates the identity and estimated diagonal specifications; J-SAA and J${}^+$-SAA use the identity specification. \\
\bottomrule
\end{tabular}
\end{table}

\subsection{Benchmark Method Details.}\label{subsec:appendix_benchmark}

\paragraph{PADR.} 
Let $\bar x=(x,1)$ denote the context augmented with an intercept. PADR parameterizes each coordinate of its raw decision as
\[
\widetilde\pi_{\mathrm{PADR},i}(x)
=
\max_{r=1,\ldots,K_1}a_{ir}^\top\bar x
-
\max_{s=1,\ldots,K_2}b_{is}^\top\bar x,
\qquad i=1,\ldots,n,
\]
where $K_1$ and $K_2$ are the numbers of max-affine pieces; when either number is zero, the corresponding term is omitted. The coefficients are estimated using the enhanced stochastic majorization-minimization algorithm of~\citet{zhang2026data}, which repeatedly selects $\epsilon$-active affine pieces and solves the resulting convex majorization subproblem. The structural and algorithmic parameters are selected using the common tuning protocol described in Section~\ref{subsec:experiment_protocol}. In the resource allocation experiments, violations of the first-stage constraints are penalized during fitting, and the resulting raw decisions are restored before the tuning-set and test-set costs are computed.

\paragraph{DD-SAA.}
We then present the details of the residual-based SAA family of methods. Let $\widehat f_{\mathrm{tr}}$ and $\widehat Q_{\mathrm{tr}}$ denote the conditional-mean and conditional-scale models fitted on $\mathcal D_{\mathrm{tr}}$, let $\widehat f^{-i}$ and $\widehat Q^{-i}$ denote the corresponding leave-one-out fits obtained after omitting observation $i$, and let $\Xi:=\mathbb R_+^{k_2}$ denote the admissible uncertainty set used to restore the residual scenarios. Define
\[
    \widehat e_i^{\mathrm{ER}}
    =
    \bigl[\widehat Q_{\mathrm{tr}}(x_i)\bigr]^{-1}
    \bigl(\xi_i-\widehat f_{\mathrm{tr}}(x_i)\bigr),
    \qquad
    \widehat e_i^{\mathrm{loo}}
    =
    \bigl[\widehat Q^{-i}(x_i)\bigr]^{-1}
    \bigl(\xi_i-\widehat f^{-i}(x_i)\bigr).
\]
The corresponding conditional scenarios are
\[
\begin{aligned}
    \widehat\xi_i^{\mathrm{ER}}(x)
    &=
    \operatorname{proj}_{\Xi}\left(
        \widehat f_{\mathrm{tr}}(x)
        +
        \widehat Q_{\mathrm{tr}}(x)\widehat e_i^{\mathrm{ER}}
    \right),\\
    \widehat\xi_i^{J}(x)
    &=
    \operatorname{proj}_{\Xi}\left(
        \widehat f_{\mathrm{tr}}(x)
        +
        \widehat Q_{\mathrm{tr}}(x)\widehat e_i^{\mathrm{loo}}
    \right),\\
    \widehat\xi_i^{J+}(x)
    &=
    \operatorname{proj}_{\Xi}\left(
        \widehat f^{-i}(x)
        +
        \widehat Q^{-i}(x)\widehat e_i^{\mathrm{loo}}
    \right).
\end{aligned}
\]
Setting $\widehat Q_{\mathrm{tr}}\equiv\widehat Q^{-i}\equiv I_{k_2}$ for every $i$ yields the homoscedastic specialization, in which the scenarios are formed by adding empirical or leave-one-out residuals to the corresponding point predictions. In all reported experiments, $\widehat f_{\mathrm{tr}}$ and $\widehat f^{-i}$ are ordinary least-squares regressions with intercepts and all $k_1$ context coordinates. J-SAA and J${}^+$-SAA use the homoscedastic specialization, whereas ER-SAA tunes between the identity scale and an estimated diagonal scale. The projection onto $\Xi$ is coordinatewise truncation at zero.
Given scenarios $\{\widehat\xi_i^a(x)\}_{i=1}^{|\mathcal D_{\mathrm{tr}}|}$ for $a\in\{\mathrm{ER},J,J+\}$, the corresponding decision satisfies
\[
    \widehat\pi_a(x)
    \in
    \arg\min_{w\in S}
    \frac{1}{|\mathcal D_{\mathrm{tr}}|}
    \sum_{i=1}^{|\mathcal D_{\mathrm{tr}}|}
    \ell\bigl(w;\widehat\xi_i^a(x)\bigr).
\]

\subsection{Computation of the LRP Optimization Layers}\label{subsec:appendix_cf}

\paragraph{Newsvendor problem.} For the one-dimensional box $S_{\mathrm{nv}}=[0,w^{\max}]$, all three regularized optimizer maps have closed-form expressions. For a latent scalar $z\in\mathbb R$ and smoothing parameter $\tau>0$, they are
\begin{align*}
    w_{I_1,\phi_{\log}}(z)
    &=
    \frac{2\tau w^{\max}}
    {w^{\max}z+2\tau+\sqrt{(w^{\max}z)^2+4\tau^2}},\\
    w_{I_1,\phi_{\mathrm{ent}}}(z)
    &=
    \frac{w^{\max}}
    {1+\exp(z/\tau)},\\
    w_{\mathrm{ptb}}(z)
    &=
    w^{\max}\Phi(-z/\tau). 
\end{align*}
The first two expressions are the minimizers induced by $\phi_{\log}$ and $\phi_{\mathrm{ent}}$, respectively. For $\mathsf{LRP}$-Ptb, the perturbed linear oracle is evaluated at $z+Z$ with $Z\sim\mathcal N(0,\tau^2)$, so its population expectation is $w^{\max}\Phi(-z/\tau)$.

\paragraph{Resource allocation with box constraints.} For $S_{\mathrm{ra}}^{\mathrm{box}}$, we take $F=I_n$, so $k=n$, and all three LRP maps are separable across resources. For a latent vector $z\in\mathbb R^n$ and each $i\in\mathcal I$, the coordinate maps are
\begin{align*}
    \bigl[w_{I_n,\phi_{\log}}(z)\bigr]_i
    &=
    \frac{2\tau w_i^{\max}}{w_i^{\max}z_i+2\tau+\sqrt{(w_i^{\max}z_i)^2+4\tau^2}},\\
    \bigl[w_{I_n,\phi_{\mathrm{ent}}}(z)\bigr]_i
    &=
    \frac{w_i^{\max}}{1+\exp(z_i/\tau)},\\
    \bigl[w_{\mathrm{ptb}}(z)\bigr]_i
    &=
    w_i^{\max}\Phi(-{z_i}/{\tau}).
\end{align*}
The first two expressions apply $\phi_{\log}$ and $\phi_{\mathrm{ent}}$ coordinatewise. For $\mathsf{LRP}$-Ptb, the perturbed box oracle selects $w_i^{\max}$ when $z_i+Z_i<0$ and zero otherwise, where $Z_i\stackrel{\mathrm{iid}}{\sim}\mathcal N(0,\tau^2)$, so the displayed formula is the exact population expectation.

\paragraph{Resource allocation with group-budget constraints.} Let $\mathcal I_1=\{1,\ldots,10\}$ and $\mathcal I_2=\{11,\ldots,20\}$. Because the group-budget constraints couple resource coordinates, the corresponding LRP maps are not coordinatewise separable. 

For $\mathsf{LRP}$-Log, we apply the logarithmic barrier directly to all box and group-budget inequalities defining $S_{\mathrm{ra}}^{\mathrm{gb}}$. The optimizer is computed by a damped Newton iteration from a strict interior point, and its vector--Jacobian product is obtained by differentiating the first-order optimality conditions. 

For $\mathsf{LRP}$-Ent, we use the extended feasible set
\[
S_{\mathrm{ra}}^{\mathrm{ext}}
=
\left\{
w^{\mathrm{ext}}=(w^\top,s_1,s_2)^\top:
0\leq w_i\leq w_i^{\max}\ \forall i\in\mathcal I,\quad
0\leq s_g\leq\bar s_g,\quad
\sum_{i\in\mathcal I_g}w_i+s_g=\bar s_g,\quad
g=1,2
\right\}.
\]
Every $w\in S_{\mathrm{ra}}^{\mathrm{gb}}$ has the unique extension $s_g=\bar s_g-\sum_{i\in\mathcal I_g}w_i$, so projection onto the first $n$ coordinates is a bijection onto $S_{\mathrm{ra}}^{\mathrm{gb}}$. This representation converts the group-budget inequalities into affine inequalities while retaining the bounded-coordinate structure used by the entropic layer. We take $F^{\mathrm{ext}} = (I_n, 0_{n\times 2})^\top$, so the latent dimension remains $k=n$, and the slack coordinates have zero linear coefficients, and the policy returns the first $n$ coordinates. Conditional on the equality multipliers $\lambda=(\lambda_1,\lambda_2)$, the bounded-variable entropic optimizer satisfies
\[
\begin{aligned}
w_i(\lambda_g)
&=
\frac{w_i^{\max}}
{1+\exp\left((z_i+\lambda_g)/\tau\right)},
&& i\in\mathcal I_g,\\
s_g(\lambda_g)
&=
\frac{\bar s_g}
{1+\exp\left(\lambda_g/\tau\right)},
&& g=1,2.
\end{aligned}
\]
The multipliers are computed by a damped Newton iteration to enforce the two equality constraints, and the vector--Jacobian product is obtained by differentiating this equality-constrained KKT system. 

For $\mathsf{LRP}$-Ptb, each perturbed linear oracle decomposes into two continuous-knapsack problems and is solved exactly by a groupwise greedy algorithm. Both the perturbation expectation and its vector--Jacobian product are estimated from the same 50 freshly drawn Gaussian perturbations using the Gaussian score estimator.

\subsection{Additional Results}\label{subsec:appendix_plot}
 Figures~\ref{fig:resource_allocation_box_boxplots} and~\ref{fig:resource_allocation_group_budget_boxplots} display the distributions underlying Tables~\ref{tab:resource_allocation_box_confirmation50} and~\ref{tab:resource_allocation_general_confirmation50}, respectively. Both figures use the same experimental settings and method ordering as the corresponding tables.

\begin{figure}[h!]
    \centering
    \includegraphics[width=\linewidth]{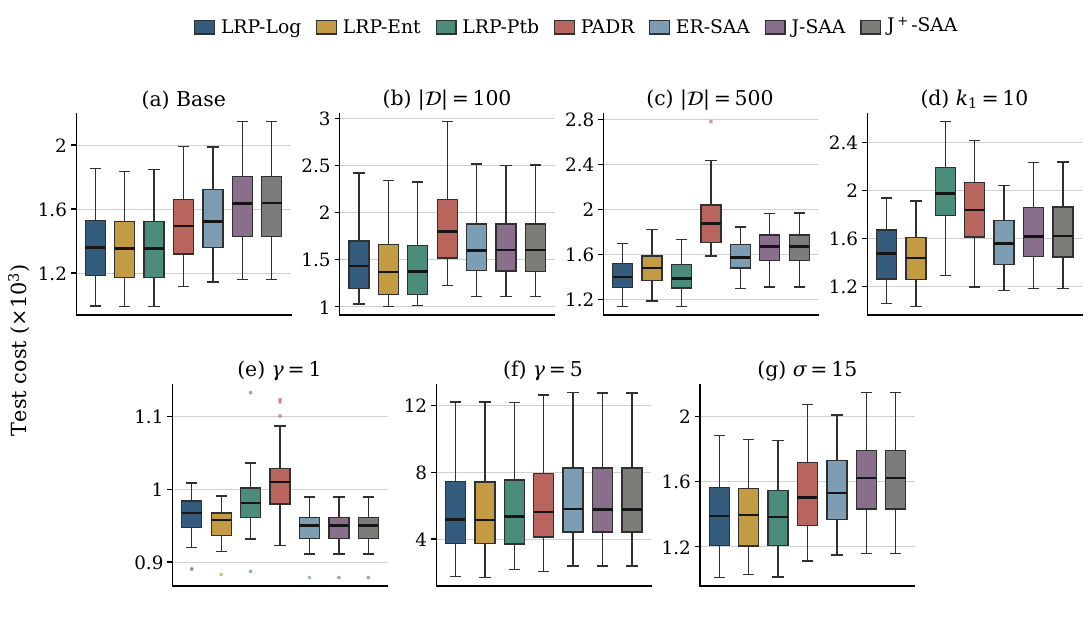}
    \caption{Out-of-sample cost distributions for the resource allocation experiment with box constraints. Each box summarizes test costs over 50 independent confirmation replications. Panel (a) shows the baseline setting $(|\mathcal D|,k_1,\gamma,\sigma)=(200,3,3,5)$; panels (b)--(g) change only the parameter indicated in the panel title while holding the remaining parameters at their baseline values. Within each panel, methods follow the legend order from left to right. Each panel uses its own vertical scale, with test costs reported in thousands. The center line is the median, the box spans the interquartile range, the whiskers extend to the most extreme observations within 1.5 interquartile ranges of the lower and upper quartiles, and points beyond the whiskers are shown individually; lower values are better.}
    \label{fig:resource_allocation_box_boxplots}
\end{figure}

\begin{figure}[h!]
    \centering
    \includegraphics[width=\linewidth]{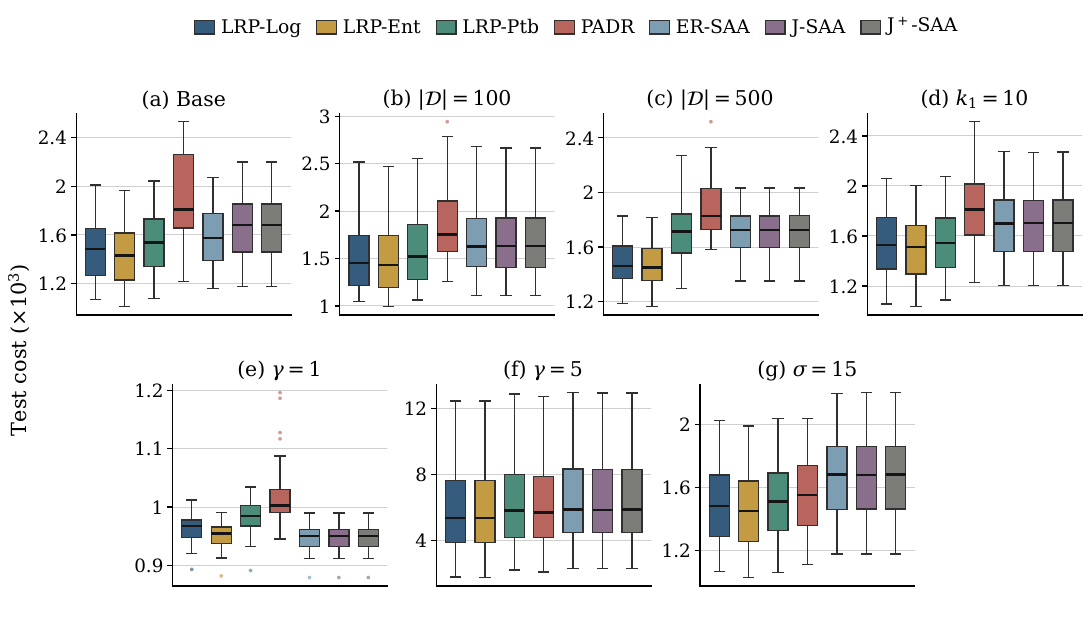}
    \caption{Out-of-sample cost distributions for the resource allocation experiment with group-budget constraints. Each box summarizes test costs over 50 independent confirmation replications. Panel (a) shows the baseline setting $(|\mathcal D|,k_1,\gamma,\sigma)=(200,3,3,5)$; panels (b)--(g) change only the parameter indicated in the panel title while holding the remaining parameters at their baseline values. Within each panel, methods follow the legend order from left to right. Each panel uses its own vertical scale, with test costs reported in thousands. The center line is the median, the box spans the interquartile range, the whiskers extend to the most extreme observations within 1.5 interquartile ranges of the lower and upper quartiles, and points beyond the whiskers are shown individually; lower values are better.}
    \label{fig:resource_allocation_group_budget_boxplots}
\end{figure}

\end{APPENDICES}

\end{document}